\documentclass{amsart}

\newtheorem{theorem}{Theorem}[section]
\newtheorem{lemma}[theorem]{Lemma}

\theoremstyle{definition}

\theoremstyle{remark}

\usepackage{graphics}
\usepackage[arrow, matrix, curve]{xy}
\usepackage{amssymb}

\newcommand{\C}{\mathbb{C}}

\newcommand{\Z}{\mathbb{Z}}

\newcommand{\Q}{\mathbb{Q}}

\numberwithin{equation}{section}

\begin{document}

\title{On modular ball-quotient surfaces\\
of Kodaira dimension one}

\author{Aleksander Momot}

\address{Departement Mathematik, ETH Z\"urich, HG J65,
R\"amistrasse 101,
8092 Z\"urich,
Switzerland}

\email{aleksander.momot@math.ethz.ch}



\begin{abstract} Let $\Gamma \subset \mathbf{PU}(2,1)$ be a lattice which is not co-compact, of finite Bergman-covolume and acting freely on the open unit ball $\mathbf{B} \subset \mathbb{C}^2$. Then the compactification $X = \overline{\Gamma \setminus \mathbf{B}}$ is a projective smooth surface with an elliptic compactification divisor $D = X \setminus (\Gamma \setminus \mathbf{B})$. In this short note we discover a new class of unramified ball-quotients $X$. We consider ball-quotients $X$ with $kod(X) = h^1(X, \mathcal{O}_X) = 1$. We prove that each minimal surfaces with finite Mordell-Weil group in the class described is up to an \'etale base change the pull-back of the elliptic modular surface which parametrizes triples $(E,x,y)$ of elliptic curves $E$ with $6$-torsion points $x,y \in E[6]$ such that $\Z x+\Z y = E[6]$.
\end{abstract}
\maketitle

\section{Introduction}
Let the symbol $\mathcal{T}$ denote the class of complex projective smooth surfaces $X$ which contain pairwise disjoint elliptic curves $D_1,...,D_{h_X}$ such that $U = X \setminus \bigcup D_i$ admits the open unit ball $\mathbf{B} \subset \C^2$ as universal holomorphic covering; as explained in \cite{Mom}, $\mathcal{T}$ forms the 'generic' class of compactified ball-quotient surfaces. There are several motivations to study surfaces in $\mathcal{T}$ without assuming that $\pi_1(U, \ast)$ with its Poincar\'e action on $\mathbf{B}$ is an arithmetic lattice of $\mathbf{PU}(2, 1)$; we refer to \cite{[Ho98]} or to the introduction of \cite{Mom}. Since the discovery of blown-up abelian surfaces in $\mathcal{T}$ by Hirzebruch and Holzapfel some years ago (cf.\,\cite{[Ho01]}) there have been no further examples of surfaces of special type in $\mathcal{T}$. In this short note we present a new class of modular surfaces $X \in \mathcal{T}$ with $kod(X) = 1$.\\
\\
In what follows we only consider complex projective smooth surfaces. Recall that a minimal elliptic surface $\pi: X \longrightarrow C$ with finite Mordell-Weil group $MW(X)$ of sections is called \textit{extremal} if $rank\,NS(X) = h^{1,1}(X)$. Particular examples arise in the following way. To each pair of positive integers 
$$(m,n) \notin \{(1,1), (1,2), (2,2), (1,3), (1,4), (2,4)\}$$
there exists a modular elliptic surface over $\overline{\mathbb{Q}}$ in the sense of Shioda \cite{Sh}
$$\pi_n(m): X_n(m) \longrightarrow C_n(m)$$
such that $\pi_n(m)$ admits no multiple fibers and has a non-constant $j$-invariant. By \cite{Sh}, $X_n(m)$ is an extremal elliptic surface with the following properties.
\begin{itemize}
 \item $MW(X_n(m)) = \Z/m\Z \times \Z/n\Z$
\item $C_n(m)$ is the (compactified) curve $\overline{\Gamma_m(n) \setminus \mathbb{H}}$ where $\Gamma_n(m) \subset \mathbf{Sl}_2(\Z)$ is the group
$$\left\lbrace \left(\begin{array}{cc} a & b \\ c & d \end{array}
 \right) ; \left(\begin{array}{cc} a & b \\ c & d \end{array}
 \right) \equiv \left(\begin{array}{cc} 1 & \ast \\ 0 & 1 \end{array}
 \right)\,mod\,m, b \equiv 0\,mod\,n \right\rbrace.$$
\item $C_n(m)$ parametrizes triples $\big((E, e_E), x,y\big)$ of elliptic curves $E$ with neutral element $e_E \in E(\C)$ and elements $x \in E[m], y \in E[n]$ such that $|\Z x + \Z y| = mn$.
\item All singular fibers of $\pi_n(M)$ are of type $I_k$ in Kodaira's notation; they lie over the cusps of $c \in C_n(m)$. A representant of $c$ in $\Q \cup \{\infty\}$ is stabilized by a matrix $\gamma \in \Gamma$ which is a $\mathbf{Sl}_2(\Z)$-conjugate of
$$\left(\begin{array}{cc} 1 & k\\0 & 1
         
        \end{array}
 \right).$$ 
\end{itemize}
The points $x$ and $y$ arise from to the intersection of $E$ with generators of $MW(X_n(m))$. More generally, by \cite[Thm.\,1.2, Thm.\,1.3]{Kl} each extremal elliptic surface ${\pi} :{X} \longrightarrow {C}$ with non-constant $j$-invariant, no multiple fibers and
$MW(\tilde{X}) = \Z/m\Z \times \Z/n\Z$, $(m,n)$ as above, allows a cartesian diagram of isogenies
$$\begin{xy} 
  \xymatrix{
  X \ar[rr]^{\pi} \ar[d] & & {{C}} \ar[d]^v\\
   X_n(m) \ar[rr]^{\pi_n(m)} & & C_n(m)
}
\end{xy}$$
With this perspective we are able to formulate our main result. We call a complex projective smooth surface $X$ \textit{irregular} if $h^1(X, \mathcal{O}_X) > 1$.
\begin{theorem} \label{t2} Let $X$ be an irregular minimal surface in $\mathcal{T}$ with $kod(X) = 1$ and finite Mordell-Weil group. Then after an \'etale base change $X$ becomes an extremal elliptic surface fibered over an elliptic curve $C$ such that the following assertions hold.
\begin{enumerate}
 \item There is a cartesian diagram over $\overline{\mathbb{Q}}$
$$\begin{xy} 
  \xymatrix{
  X \ar[rr]^{\pi} \ar[d] & & {{C}} \ar[d]\\
   X_{6}(6) \ar[rr]^{\pi_{6}(6)} & & C_{6}(6)
}
\end{xy}$$
\item The compactification divisor $D$ of $X$ consists of 36 sections of $\pi$, each having self-intersection number $-\chi(X)$. The fibration $\pi$ admits $2\chi(X)$ singular fibers of type $I_6$, and each component of an $I_6$ intersects $D$ in precisely $6$ points. We have $rank\,NS(X) = 10\chi(X) + 2$.  
\end{enumerate}
Conversely, $X_{6}(6)$ is an extremal elliptic and irregular surface in $\mathcal{T}$. 
\end{theorem}
The diagram in the first statetement is induced by the $j$-invariant.
\section{Some basic properties of surfaces in $\mathcal{T}$} We cite two results on ball-quotient surfaces which will be needed for the proof of the theorem. The first result is essentially \cite[Thm.\,3.1]{TY} specified to $\dim\,X = 2$ with attention to sign conventions, except the assertion on semi-stability. The latter assertion follows from \cite{Miy}. A reduced effective divisor is called \textit{semi-stable} if it has normal crossings and if every rational smooth prime component intersects the remaining components in more than one point.
\begin{theorem} [Tian-Yau/Miyaoka-Sakai]\label{TY} Let $X$ be a smooth projective 
surface and $D \subset X$ a divisor with normal crossings. Suppose that $K_X + 
D$ is big and ample modulo D. Then
$$c_1^2(\Omega_X^1(log\,D)) \leq 3 c_2(\Omega_X^1(log\,D)),$$
with equality holding if an only if $X\setminus D$ is an unramified ball quotient $\Gamma \setminus \mathbf{B}$ and $D$ is semi-stable.                                     
\end{theorem} 
There is a canonical exact sequence
$$0 \longrightarrow \Omega^1_X \longrightarrow \Omega^1_X(log\,D) 
\stackrel{res}{\longrightarrow} \mathcal{O}_D \longrightarrow 0$$
where $res$ is the Poincar\'e residue map. With this one proves that $c_1(\Omega^1_X(log\,D)) = [D] - c_1(X) \in H^2(X, \mathbb{C})$ and $c_2(\Omega^1_X(log\,D)) = c_2(X)-(c_1(X),[D]) + ([D],[D]) \in H^4(X, \mathbb{C})$. Therefore,
$$c_1^2(\Omega^1_X(log\,D)) = (K_X+D)^2.$$
It is interesting to note that if equality holds in the theorem, then $D$ is smooth. Namely, if $\Gamma' \subset \Gamma$ is a neat normal subgroup with finite index in $\Gamma$, then $\Gamma'\setminus \mathbf{B}$ is compactified by a smooth elliptic divisor, and $\Gamma \setminus \mathbf{B}$ is compactified by a divisor $D$. As $D$ is the quotient $D'/G$, $G = \Gamma/\Gamma'$, it is a normal curve. Hence, $D$ is smooth and consists of elliptic curves, for rational curves cannot appear because of semi-stability.
The next is proved \textit{verbatim} as \cite[Lemma 3.2]{Mom}.
\begin{lemma}\label{5.5.} Let $X$ be in $\mathcal{T}$ with compactification divisor
$D$ and consider an irreducible curve $L \subset X$. If $L$ is smooth rational then $|L \cap D| \geq 3$. If $L$ is a smooth elliptic curve then $|L \cap D| \geq 1$.\end{lemma}

\section{Proof of the results} General theory asserts that $X$ admits an elliptic fibration $\pi: X \longrightarrow C$ which is the Albanese morphism. As $K_X + D$ is ample modulo $D$, it follows that a general fiber $F$ has positive self-intersection with $D$. Thus, a component of $D$ dominates $C$. Hence, $C$ is an elliptic curve and $h^1(X, \mathcal{O}_X) = 1$. Moreover, after transition to an etale cover $\tilde{C}$ of $C$ and performing a base change, we can achieve that every $D_i$ is a section, as soon as it dominates $C$ (\cite[Lemma 3.3]{Mom}). We will assume this from now on. Since the curves $D_i$ are pair-wise disjoint, they must be all sections.
\begin{lemma} \label{efg} The identities $36\chi(X) = DF \cdot \chi(X) = -D^2$ and $DF = 36$ hold.
\end{lemma}
\begin{proof} The canonical bundle-formula implies that $K_X = \pi^{\ast}(\mathfrak{c})$ with a divisor Weil divisor $\mathfrak{c} \in Div(C)$. Moreover, $h^0(X, mK_X) = h^0(C, m\mathfrak{c})$. The theorem of Riemann-Roch yields $h^0(X, K_X) = \deg\,\mathfrak{c} > 0.$ Adjunction formula implies that 
$$D_i^2 = - \deg\,\mathfrak{c} = - h^0(X, K_X) = -\chi(X).$$
Hence, $-D^2 = -\sum D_i^2 = DF\chi(X)$. Furthermore, $12\chi(X) = c_2(X)$ by Noether's formula. So, Thm.\,\ref{TY} yields the remaining identities.
\end{proof}
We consider the Mordell-Weil group $MW(X) = MW_{tor}(X)$. It follows that $|MW_{tor}(X)| \geq 36.$ We prove the following lemma of general interest.
\begin{lemma} Let $\pi: X \longrightarrow C$ be a minimal elliptic surface over an elliptic curve $C$ and assume that $kod(X) \geq 1$ and that each rational curve $L \subset X$ meets at least three sections of $\pi$. Suppose moreover that $D = MW_{tor}(X) \geq 33$. Then all singular fibers of $\pi$ are semi-stable of type $I_6$, $X$ has $2\chi(X)$ singular fibres and $MW(X) = MW_{tor}(X) = \Z/6\Z \times \Z/6\Z$ and the rank of the Neron severi group $NS(X)$ equals $h^{1,1}(X) = 10\chi(X) + 2$.   
\end{lemma}
\begin{proof} The assertion concerning $MW(X)$ follows directly from \cite[(4.8)]{MP}. \cite[Lemma 1.1]{MP} implies then that all singular fibers are of type $I_n$. If $H_n \subset M(X)$ is the non-trivial isotropy group of a node $x \in I_n$ then $MW_{tor}(X)/H_n$ is cyclic by \cite[Lemma 2.2]{MP}. Moreover, all nodes from one and the same fiber have the same isotropy group by \cite[Lemma 2.1, (c)]{MP}, and this isotropy group is non-trivial by \cite[Lemma 2.1, (b)]{MP} and because a component of $I_n$ meets at least three sections. Thus, always $|H_n| \geq 6$. On the other hand, by \cite[p.\,251]{MP} and \cite[Lemma 2.3, (f)]{MP}, $\sum_{I_n} n = c_2(X)$ and 
$$36c_2(X) = |MW_{tor}(X)|c_2(X) = \sum_{I_n} n|H_n|^2.$$
Hence, always $|H_n| = 6$. Let $S \in MW(X)$ be the neutral element. By the proof of \cite[Lemma 2.2]{MP}, $H_n$ consists of precisely those sections meeting the prime component $L \subset I_n$ which contains $S \cap I_n$. However, since we may take any section to be the neutral element of $MW(X)$, for each component $L \subset I_n$ we have $LD = 6$. As $DI_n = 36$, we get $n = 6$. Finally, recalling that $\sum_{I_n} n = c_2(X)$, we find for the number $t$ of singular fibers:

 $$t = 2\chi(X) = 2g(C)-2 + rank\,MW(X) + 2\chi(X).$$

According to \cite[Prop.\,1.6]{MP} this equality holds if and only if $rank\,NS(X) = h^{1,1}(X)$. An easy calculation shows now that $h^{1,1}(X) = 10\chi(X) + 2$.
\end{proof}
It follows that $X$ is isomorphic to a pull-back $X_{6}(6) \times_{C_6(6)}C$.


\begin{thebibliography}{99}


\bibitem{[Ho98]}
Holzapfel, R.-P., \textit{Ball and Surface Arithmetics}, Aspects vol. \textbf{E29} (Vieweg, Braunschweig, 1998)

\bibitem{[Ho01]} Holzapfel, R.-P., \textit{Jacobi theta embedding of a hyperbolic 4-space with cusps}, in: Geometry, integrability and quantization (Coral\,Press\,Sci.\,Publ., Sofia, 2002).  
\bibitem{kod} Kodaira, K., \textit{On stability of compact submanifolds of compact complex manifolds}, Am.\,J.\,Math. vol.\,\textbf{85} (1963), 79-94  

\bibitem{MP} Miranda, R.; Persson, U., \textit{Torsion groups of elliptic surfaces}, Comp.\,Math.\,tome \textbf{72} (1989), 249-267

\bibitem{Miy} Miayoka, Y., \textit{The maximal number of quotient singularities on surfaces with given numerical invariants}, Math\,Ann.\,\textbf{268} (1984) 

\bibitem{Mo} Mostow, G.D., \textit{Strong Rigidity of Locally Symmetric Spaces}, Annals of Mathematics Studies, No. \textbf{78} ( Princeton University Press, Princeton, N.J.; University of Tokyo Press, Tokyo, 1973)

\bibitem{Mom} Momot, A., \textit{Irregular ball-quotient surfaces with non-positive Kodaira dimension}, to appear

\bibitem{Kl} Kloostermann, R., \textit{Extremal elliptic surfaces and infinitesimal Torelli}, Michigan Math.\,J.\,52 (2004)

\bibitem{TY} Tian, G., Yau, S.T., \textit{Existence of K\"ahler-Einstein metrics on complete K\"ahler manifolds and their applications to algebraic geometry}, in: Yau (ed.), Mathematical Aspects of String Theory, Advanced Series in Mathematical Physics vol.\textbf{1} (World Scientific, Singapore, 1987) 
  
\bibitem{Sh} Shioda, T., \textit{On elliptic modular surfaces}, J.\,Math.\,Soc.\,Japan 24 (1972), 20-59

\end{thebibliography}
\end{document}